\documentclass[a4paper,twoside,11pt]{amsart}

\usepackage{graphicx}
\usepackage[dvips]{color}
\usepackage[all]{xy}
\usepackage[english]{babel}

\usepackage{amsmath,amssymb}

\usepackage{amsthm}

\usepackage{amscd}

\usepackage{amsfonts}

\usepackage{mathrsfs}

\begin{document}

\title{Compactified Picard stacks over $\overline {\M}_g$}
\author{Margarida Melo}\thanks{Partially supported by a Funda\c{c}\~ao Calouste
Gulbenkian fellowship.}

\newcommand{\av}{``}
\newcommand{\uv}{''}
\newcommand{\forn}{\forall n\in\mathbb{N}}
\newcommand{\raw}{\to}
\newcommand{\Raw}{\Rightarrow}
\newcommand{\law}{\gets}
\newcommand{\Law}{\Leftarrow}
\newcommand{\Lra}{\Leftrightarrow}
\newcommand{\C}{\mathcal C}
\newcommand{\K}{\mathcal K}
\newcommand{\A}{\mathcal A}
\newcommand{\N}{\mathbb N}
\newcommand{\X}{\mathcal X}
\newcommand{\M}{\mathcal M}
\newcommand{\Z}{\mathbb Z}
\renewcommand{\P}{\mathbb P}
\newcommand{\ra}{\rightarrow}

\def\hifen{\discretionary{-}{-}{-}}

\newtheorem{teo}{Theorem}[section]
\newtheorem{prop}[teo]{Proposition}
\newtheorem{lem}[teo]{Lemma}
\newtheorem{rem}[teo]{Remark}
\newtheorem{defi}[teo]{Definition}
\newtheorem{ex}[teo]{Example}
\newtheorem{exs}[teo]{Examples}
\newtheorem{cor}[teo]{Corollary}
\pagestyle{myheadings}
\markboth{\small{M. Melo}}{\small{\textit{Compactified Picard stacks over $\overline {\M}_g$}}}

\maketitle

\begin{abstract}
We study algebraic (Artin) stacks over $\overline{\mathcal M}_g$
giving a functorial way of compactifying the relative degree $d$
Picard variety for families of stable curves. We also describe for
every $d$ the locus of genus $g$ stable curves over which we get
Deligne-Mumford stacks strongly representable over
$\overline{\mathcal M}_g$.
\end{abstract}

\renewcommand{\theequation}{\arabic{section}.\arabic{equation}}

\section{Introduction}

In this paper we study (compactified) moduli stacks of line bundles
over the moduli stack of stable curves $\overline{\mathcal M}_g$.
The aim is to give a functorial way of getting compactified Picard
varieties (of degree $d$) for families of stable curves. In other
words, we study geometrically meaningful algebraic stacks
$\overline{\mathcal P}_{d,g}$ with a map to $\overline{\mathcal
M}_g$ such that, given a family of stable curves $f:\mathcal
X\rightarrow S$, the fiber product of $\overline{\mathcal
P}_{d,g}\rightarrow \overline{\mathcal M}_g$ by the moduli map
$\mu_f:S\rightarrow \overline{\mathcal M}_g$ is either a
compactification of the relative degree $d$ Picard variety
associated to $f$ or has a canonical map onto it (see below).

There are many constructions of compactified Picard varieties of
stable curves. We will choose the one built by Caporaso in
\cite{cap}. This compactification, $\overline P_{d,g}$, is
constructed as a GIT-quotient and has a proper morphism
$\phi_d:\overline P_{d,g}\rightarrow \overline M_g$ such that
$\phi_d^{-1}(M_g^0)$ is isomorphic to the \av universal Picard
variety of degree $d$\av, Pic$_g^d$, parametrize isomorphism
classes of line bundles of degree $d$ over automorphism-free
nonsingular curves. Points in $\overline P_{d,g}$ correspond to
isomorphism classes of balanced line bundles of degree $d$ in
quasistable curves of genus $g$ (see Definition \ref{balanced}). In
particular, given $[X]\in \overline M_g^0$, the smooth locus of
$\phi_d^{-1}(X)$ is isomorphic to the disjoint union of a finite
number of copies of the jacobian of $X$, $J_X$.

Let $f:\mathcal X\rightarrow S$ be a family of (genus $g$) stable
curves. By a compactification of the relative Picard variety of
degree $d$ associated to $f$ we mean a projective $S$-scheme $P$
whose fiber over closed points $\xi$ of $S$ is isomorphic to
$\phi_d^{-1}(X_\xi)$ ($X_{\xi}$ denotes the fiber of $f$ over
$\xi$).

Let $(d-g+1,2g-2)=1$. Then, the GIT-quotient yielding $\overline
P_{d,g}$ is geometric (see \cite{cap}, Prop. 6.2) and the quotient
stack associated to it, $\overline{\mathcal P}_{d,g}$, is a
Deligne-Mumford stack with a strongly representable morphism onto
$\overline{\mathcal M}_g$ (see \cite{capneron}, 5.9). So, given a
family of stable curves $f:\mathcal X\rightarrow S$, the base change
of the moduli map $\mu_f:S\rightarrow \overline{\mathcal M}_g$ by
$\overline{\mathcal P}_{d,g}\rightarrow \overline{\mathcal M}_g$ is
a scheme, $\overline P_f^d$, yielding a compactification of the
relative degree $d$ Picard variety associated to $f$. Moreover,
$\overline{\mathcal P}_{d,g}$ parametrizes N\'eron models of
Jacobians in the following way. Let $B$ be a smooth curve defined
over an algebraically closed field $k$ with function field $K$ and
$\mathcal X_K$ a smooth genus $g$ curve over $K$ whose regular
minimal model over $B$ is a family $f:\mathcal X\rightarrow B$ of
stable curves. Then, the smooth locus of the map $\overline
P_f^d\rightarrow B$ is isomorphic to the N\'eron model of
Pic$^d\mathcal X_k$ over $B$ (see Theorem 6.1 of loc. cit.).

If $(d-g+1,2g-2)\neq 1$, the quotient stack $\overline{\mathcal P}_{d,g}$ is not Deligne-Mumford. In particular, its natural map onto $\overline{\mathcal M}_g$ is not representable. In section \ref{dgeneralstack} we consider the restriction of $\overline{\mathcal P}_{d,g}$ to the locus $\overline{\mathcal M}_g^d$ of $d$-general curves, i.e., the locus of genus $g$ stable curves over which the GIT-quotient above is geometric. In Prop. \ref{propdm} we show that this restriction, denoted by $\overline{\mathcal P}_{d,g}^{Ner}$, is a Deligne-Mumford stack and is endowed with a strongly representable map onto $\overline{\mathcal M}_g^d$. So, it gives a functorial way of getting a compactification of the relative degree $d$ Picard variety for families of $d$-general curves, generalizing Prop. 5.9 of \cite{capneron}.

In section \ref{dgeneraldesc} we give a combinatorial description of the locus in $\overline{M}_g$ of $d$-general curves, $\overline{M}_g^d$. For each $d$, $\overline{M}_g^d$ is an open subscheme of $\overline{M}_g$ containing all genus $g$ irreducible curves and $\overline M_g^d=\overline{M}_g^{d^\prime}$ if and only if $(d-g+1,2g-2)=(d^\prime -g+1,2g-2)$.
In Prop. \ref{lattice} we also show that the $\overline{M}_g^d$ yield a lattice of open subschemes of $\overline M_g$ parametrized by the (positive) divisors of $2g-2$.

Let us now consider a family $f:\mathcal X\rightarrow S$ of genus $g$ stable curves which are not $d$-general. Then the fiber product of the moduli map of $f$, $\mu_f:S\rightarrow \overline{\mathcal M}_g$ by the map $\overline{\mathcal P}_{d,g}\rightarrow \overline{\mathcal M}_g$
is not a scheme. Indeed, it is not even an algebraic space. The best we can get here is that it is canonically endowed with a proper map onto a scheme yielding a compactification of the relative degree $d$ Picard variety associated to $f$ (see Prop. \ref{propcan}).

In \cite{capneron} 5.10, a modular description of the quotient stack $\overline{\mathcal P}_{d,g}$ is given in the case $(d-g+1,2g-2)=1$.
In section \ref{modular} we search for a modular description of $\overline{\mathcal P}_{d,g}$ for any $d$. The main difficulty here is that in the general case it is not possible to construct the analogue of Poincar\'e line bundles for families of stable curves, fundamental in the modular description given in loc. cit, 5.10. To overcome this difficulty we will study the moduli stack of balanced line bundles of relative degree $d$ for families of quasistable curves of genus $g$, defined in \ref{modulardef} and denoted by $\overline{\mathcal G}_{d,g}$.
In Theorem \ref{geomdesc} we show that $\overline{\mathcal G}_{d,g}$ is an algebraic stack with a morphism onto $\overline{\mathcal M}_g$. However, there is an action of $\mathbb G_m$ on $\overline{\mathcal G}_{d,g}$ compatible with the map to $\overline{\mathcal M}_g$, so it cannot be representable over $\overline{\mathcal M}_g$. In section \ref{rigidification} we give a modular description of the rigidification (defined by Abramovich, Corti and Vistoli in \cite{avc}) of $\overline{\mathcal G}_{d,g}$ along the action of $\mathbb G_m$ and we show that it is isomorphic to $\overline{\mathcal P}_{d,g}$. Moreover, $\overline{\mathcal G}_{d,g}$ is a $\mathbb G_m$-gerbe over $\overline{\mathcal P}_{d,g}$.

Finally, we ask if, for every $d$, $\overline{\mathcal P}_{d,g}$,
parametrizes N\'eron model of Jacobians of smooth curves as it does
if $(d-g+1,2g-2)=1$.

In the last section we show that the answer is no if
$(d-g+1,2g-2)\neq 1$, essentially because $\overline{\mathcal
P}_{d,g}$ is not representable over $\overline{\mathcal M}_g$. We
here focus on the case $d=g-1$, which is particularly important. In
fact for this degree all known compactified Jacobians are
canonically isomorphic and are endowed with a theta divisor which is
Cartier and ample (see \cite{alex}).

\subsection{Preliminaries and notation}

We will always consider schemes and algebraic stacks locally of finite type over an algebraically closed base field $k$.

A curve $X$ will always be a connected projective curve over $k$ having at most nodes as singularities. We will denote by $C_1, \dots, C_\gamma$ the irreducible components of $X$.

\subsubsection{Line bundles on reducible curves}

We will denote by $\omega_X$ the canonical or dualizing sheaf of $X$.
For each proper subcurve $Z$ of $X$ (which we always assume to be complete), denote by $Z^\prime:=\overline {X\setminus Z}$, by $k_Z:=\sharp (Z\cap Z^\prime)$ and by $g_Z$ its arithmetic genus.
Recall that, if $Z$ is connected, the adjunction formula gives
\begin{equation}\label{omegaZ}
w_Z:=\mbox{deg}_Z\omega_X=2g_Z-2+k_Z.
\end{equation}

For $L\in $ Pic$X$ its \textit{multidegree} is \underline{deg}$L:=(\mbox{deg}_{C_1}L,\dots,$deg$_{C_\gamma}L)$ and its (total) degree is deg$L:=$deg$_{C_1}L+\dots +$deg$_{C_\gamma}L$.

Given $\underline{d}=(d_1,\dots, d_\gamma)\in \Z^\gamma$, we set Pic$^{\underline d}X:=\{L\in$Pic$X:$ \underline{deg}$L=\underline{d}\}$ and Pic$^dX:=\{L\in$Pic$X:$deg$L=d\}$. We have that Pic$^d X=\sum_{|\underline d |=d}$Pic$^{\underline d}X$, where $|\underline d|=\sum_{i=1}^\gamma d_i$.

The \textit{generalized jacobian} of $X$ is $$\mbox{Pic}^{\underline 0}\mbox{ X}=\{L\in \mbox{ Pic }X: \underline{\mbox{deg }}L=(0,\dots,0)\}.$$

\subsubsection{The relative Picard functor} \label{picard}

Let $X$ be an $S$-scheme with structural morphism $\pi:X\rightarrow S$. Given another $S$-scheme $T$, we will denote by $\pi_T:X_T\rightarrow T$ the base-change of $\pi$ under the structural morphism $T\rightarrow S$.
\begin{equation*}
\xymatrix{
{X_T:=T\times_SX} \ar[d]_{\pi_T} \ar[r] & {X} \ar[d]^{\pi} \\
{T} \ar[r] & {S}
}
\end{equation*}

By a family of nodal curves we mean a proper and flat morphism of schemes over $k$, $f:\mathcal X \rightarrow B$, such that every closed fiber of $f$ is a connected nodal curve.

We will denote by $\mathcal Pic_{f}$ the \textit{relative Picard
functor} associated to $f$ and by $\mathcal Pic_f^d$ its subfunctor
of line bundles of relative degree $d$. $\mathcal Pic_f$ is the
fppf-sheaf associated to the functor $\mathcal
P$:SCH$_B\rightarrow$Sets which associates to a scheme $T$ over $B$
the set Pic$(\mathcal X_T)$. In particular, if the family $f$ has a
section, $\mathcal Pic_f(T)=$Pic$(\mathcal X_T) /$Pic$(T)$ (see
\cite{BLR}, chapter 8 for the general theory about the construction
of the relative Picard functor).

Thanks to more general results of D. Mumford and A. Grothendieck in \cite{mumford} and \cite{SGA}, we know that $\mathcal Pic_f$ (and also $\mathcal Pic_f^d$) is representable by a scheme Pic$_f$, which is separated if all geometric fibers of $f$ are irreducible (see also \cite{BLR}, 8.2, Theorems 1 and 2).
Pic$_g^d$, the \av universal degree $d$ Picard variety\uv, coarsely represents the degree $d$ Picard functor for the universal family of (au\-to\-mor\-phism-free) nonsingular curves of genus $g$, $f_g:\mathcal Z_g\rightarrow M_g^0$.
Furthermore, it was proved by Mestrano and Ramanan in \cite{mestrano} for char $k$=0 and later on by Caporaso in \cite{cap} for any characteristic that Pic$_g^d$ is a fine moduli space, that is, there exists a Poincar\'e line bundle over Pic$_g^d\times_{M_g^0}\mathcal Z_g$, if and only if the numerical condition $(d-g+1,2g-2)=1$ is satisfied.

\subsubsection{Stable and semistable curves}

A \textit{stable} curve is a nodal connected curve of genus $g\ge2$ with ample dualizing sheaf. We will denote by $\overline M_g$ (resp. $\overline \M_g$) the moduli scheme (resp. stack) of stable curves and by $\overline M_g^0\subset \overline M_g$ the locus of curves with trivial automorphism group.

A \textit{semistable} curve is a nodal connected curve of genus $g\ge 2$ whose dualizing sheaf has non-negative multidegree.

A nodal curve $X$ is stable (resp. semistable) if, for every smooth
rational component $E$ of $X$, $k_E\ge 3$ (resp. $k_E\ge 2$.) If $X$
is semistable, the smooth rational components $E$ such that $k_E=2$
are called \textit{exceptional}.

A semistable curve is called \textit{quasistable} if two exceptional
components never meet.

The \textit{stable model} of a semistable curve $X$ is the stable curve obtained by contracting all the exceptional components of $X$.

A \textit{family of stable} (resp. \textit{semistable},
resp.\textit{quasistable}) curves is a flat projective morphism
$f:\mathcal X\rightarrow B$ whose geometrical fibers are stable
(resp. semistable, resp. quasistable) curves. A line bundle of
degree $d$ on such a family is a line bundle on $\X$ whose
restriction to each geometric fiber has degree $d$.

\subsection{Balanced line bundles over semistable curves}

Recall that Gie\-se\-ker's construction of $\overline M_g$ consists of a GIT-quotient of the action of $PGL(N)$ on a Hilbert scheme where it is possible to embed all semistable curves of genus $g$ (the "Hilbert point" of the curve) (see \cite{gieseker}). Gieseker shows that in this Hilbert scheme, in order for the Hilbert point of a curve to be GIT-semistable, it is necessary that the multidegree of the line bundle giving its projective realization must satisfy an inequality, called the "Basic Inequality". Later, in \cite{cap}, Caporaso shows that this condition is also sufficient.

We will now give the definition of this inequality, extending the terminology introduced in \cite{ccc}.

\begin{defi} \label{balanced}
Let $X$ be a semistable curve of genus $g\ge 2$ and $L$ a degree $d$ line bundle on $X$.
\begin{itemize}
\item[(i)] We say that $L$ (or its multidegree) is semibalanced if, for every connected proper subcurve $Z$ of $X$ the following (\av Basic Inequality\av) holds
\begin{equation}\label{basic}
m_Z(d):=\frac{d w_Z}{2g-2}-\frac {k_Z}2\le \mbox{deg}_ZL\le\frac{d w_Z}{2g-2}+\frac{k_Z}2:=M_Z(d).
\end{equation}
\item[(ii)] We say that $L$ (or its multidegree) is balanced if it is semibalanced and if deg$_EL=1$ for every exceptional component $E$ of $X$.
The set of balanced line bundles of degree $d$ of a curve $X$ is denoted by $B_X^d$.
\item[(iii)] We say that $L$ (or its multidegree) is stably balanced if it is ba\-lan\-ced and if for each connected proper subcurve
$Z$ of $X$ such that deg$_ZL=m_Z(d)$, the complement of $Z$, $Z^\prime$, is a union of exceptional components.

The set of stably balanced line bundles of degree $d$ on $X$ will be denoted by $\tilde B_X^d$.
\end{itemize}
\end{defi}

\begin{rem}\upshape{
Balanced multidegrees are representatives for multidegree classes of line bundles on $X$ up to twisters (that is, to elements in the degree class group of $X$, $\Delta_X$, which is a combinatorial invariant of the curve defined in \cite{cap}). More particularly, in \cite{capneron}, Proposition 4.12, Caporaso shows that, if $X$ is a quasistable curve, every multidegree class in $\Delta_X$ has a semibalanced representative and that a balanced multidegree is unique in its equivalence class if and only if it is stably balanced.}
\end{rem}

We now list some easy consequences of the previous definition.

\begin{rem}\label{rembalanced}\upshape{
\begin{enumerate}

\item[(A)]
If a semistable curve $X$ admits a balanced line bundle $L$, then $X$ must be
quasistable.
\item[(B)] To verify that a line bundle $L$ is balanced it is enough to check that deg$_ZL\ge m_Z(d)$, for each proper subcurve $Z$ of $X$ and that deg$_EL=1$ for each exceptional component $E$ of $X$.
\item[(C)] If $X$ is a stable curve, then a balanced line bundle $L$ on $X$ is stably balanced if and only if, for each proper connected subcurve $Z$ of $X$, deg$_ZL\neq m_{Z}(d)$.
\item[(D)] Let $X$ be a stable curve consisting of two irreducible components, $Z$ and $Z^\prime$, meeting in an arbitrary number of nodes. Then $X$ admits a degree $d$ line bundle which is balanced but not stably balanced if and only if $\frac{d-g+1}{2g-2}w_Z\in\Z$ (equivalently if $\frac{d-g+1}{2g-2}w_{Z^\prime}\in \Z$).
\item[(E)] A line bundle is balanced (resp. stably balanced) if and only if $L\otimes \omega_X^{\otimes n}$ is balanced (resp.stably balanced), for $n\in \Z$. So, given integers $d$ and $d^\prime$ such that $\exists n\in \Z$ with $d\pm d^\prime=n(2g-2)$, there are natural isomorphisms $B_X^d\cong B_X^{d^\prime}$ (and $\tilde B_X^d\cong \tilde B_X^{d^\prime}$).
\end{enumerate}}
\end{rem}

For (A) and (B) see \cite{capest} Remark 3.3. (C) and (E) are immediate consequences of the definition. For (D) note that, given a balanced $\gamma$-uple $\underline d\in \Z^\gamma$ such that $|\underline d|=d$, there exists a (balanced degree $d$) line bundle $L$ in $X$ such that \underline{deg}$L=\underline d$. Since $k_Z=w_Z-2g_Z+2$, we can write $m_Z(d)$ as $\frac{d-g+1}{2g-2}w_Z+g_Z-1$, which is an integer by hypothesis. In the same way $m_{Z^\prime}(d)=d-m_Z(d)$ is an integer too, so $(m_Z(d),m_{Z^\prime}(d))$ is a balanced multidegree which is not stably balanced.

\subsection{The compactified Picard variety of degree $d$ over $\overline M_g$}

Let $\overline{P}_{d,g}\rightarrow \overline{M}_g$ be Caporaso's compactification of the universal Picard variety of degree $d$, Pic$_g^d\rightarrow M_g^0$, constructed in \cite{cap}.

For $d>>0$, $\overline P_{d,g}$ is the GIT-quotient
$$\pi_d:H_d\rightarrow H_d/PGL(r+1)=:\overline P_{d,g}$$
where $H_d=(Hilb_{\P^r}^{dt-g+1})^{ss}$, the locus of GIT-semistable
points in the Hilbert scheme $Hilb_{\P^r}^{dt-g+1}$, which is
naturally endowed with an action of $PGL(r+1)$ leaving $H_d$
invariant. $\overline{P}_{d,g}$ naturally surjects onto
$\overline{M}_g$ via a proper map
$\phi_d:\overline{P}_{d,g}\rightarrow \overline{M}_g$ such that
$\phi_d^{-1}(\overline{M}_g^0)$ is isomorphic to Pic$_g^d$.

For $[X]\in \overline M_g$, denote by $\overline{P}_{d,X}$ the inverse image of $X$ by $\phi_d$. $\overline P_{d,X}$ is a connected projective scheme having at most $\Delta_X$ irreducible components, all of dimension $g$. In addition, if $X$ is automorphism-free, the smooth locus of $\overline{P}_{d,X}$ is isomorphic to the disjoint union of a finite number of copies of $J_X$.

Points in $H_d$ correspond to nondegenerate quasistable curves in $\P^r$ embedded by a balanced line bundle.

Let $H_d^s\subseteq H_d$ be the locus of GIT-stable points. These
correspond to nondegenerate quasistable curves in $\mathbb P^r$
embedded by a stably balanced line bundle of degree $d$.

\begin{defi}\label{dgeneral}
Let $X$ be a semistable curve of arithmetic genus $g\ge 2$. We say that $X$ is $d$-general if all degree $d$ balanced line bundles on $X$ are stably balanced. Otherwise, we will say that $X$ is $d$-special.
\end{defi}

Denote by $U_d:=(\phi_d\circ\pi_d)^{-1}(\overline{M}_g^d)$ the subset of $H_d$ corresponding to $d$-general curves. $U_d$ is an open subset of $H_d$ where the GIT-quotient is geometric (i.e., all fibers are $PGL(r+1)$-orbits and all stabilizers are finite and reduced), invariant under the action of $PGL(r+1)$.

$U_d=H_d$ if and only if $(d-g+1,2g-2)=1$, so the GIT-quotient yielding $\overline{P}_{d,g}$ is geometric if and only if $(d-g+1,2g-2)=1$ (see Prop. 6.2 of loc. cit.).

\section{Balanced Picard stacks on $d$-general curves}\label{dgeneralstack}

By reasons that will be clear in a moment (see Section \ref{neron1}), call $\overline P_{d,g}^{Ner}$ the GIT-quotient of $U_d$ by $PGL(r+1)$.

For the time being let $$G:=PGL(r+1).$$

Let us now consider the quotient stack
$$\overline{\mathcal P}_{d,g}^{Ner}:=[U_d/G].$$

Recall that, given a scheme $S$ over $k$,a section of
$\overline{\mathcal P}_{d,g}^{Ner}$ over $S$ consists of a pair
$(\phi:E\rightarrow S, \psi:E\rightarrow U_d)$ where $\phi$ is a
$G$-principal bundle and $\psi$ is a $G$-equivariant morphism.
Arrows correspond to those pullback diagrams which are compatible
with the morphism to $U_d$.

Let $\overline{\mathcal M}_g^d\subset \overline{\mathcal M}_g$ be
the moduli stack of $d$-general stable curves. There is a natural
map from $\overline{\mathcal P}_{d,g}^{Ner}$ to $\overline{\mathcal
M}_g^ d$, the restriction to $d$-general curves of the moduli stack
of stable curves, $\overline{\mathcal M}_g$.
In fact, the restriction to $U_d$ of the stabilization morphism from 
$H_d$ to $\overline{M}_g$ factors through $\overline{M}_g^d$ and,
since $U_d$ is invariant under the action of $G$, this yields a map
from $\overline{\mathcal P}_{d,g}^{Ner}$ to $\overline{M}_g^d$.

Recall the following definitions.

\begin{defi}
We say that a morphism of stacks $f:\mathcal F \rightarrow \mathcal G$ is \textit{representable} (resp. \textit{strongly representable}) if for any scheme $Y$ with a morphism $Y\rightarrow \mathcal G$ the fiber product $\mathcal F \times_\mathcal G Y$ is an algebraic space (resp. a scheme).
\end{defi}

Note that morphisms of schemes are always strongly representable.

\begin{defi}
A \textit{coarse moduli space} for an stack $\mathcal F$ is an algebraic space $F$ together with a morphism $\pi:\mathcal F\rightarrow F$ satisfying the following properties:
\begin{enumerate}
 \item 
 for any algebraically closed field $\Omega$, $\pi$ induces an isomorphism between the connected components of the grou\-poids $\mathcal F(\mbox{Spec }\Omega)$ and $F(\mbox{Spec }\Omega)$;
\item $\pi$ is universal for morphisms from $\mathcal F$ onto algebraic spaces.
\end{enumerate}
\end{defi}

\begin{ex}\label{coarse}\upshape{
GIT-geometric quotients by the action of an algebraic group in a scheme are coarse moduli spaces for the quotient stack associated to that action (see \cite{vistoli} 2.1 and 2.11).}
\end{ex}

\begin{lem}\label{rep}
Let $f:\mathcal F\rightarrow \mathcal G$ be a representable morphism of Deligne-Mumford stacks admitting coarse moduli spaces $F$ and $G$, respectively. Then, if the morphism induced by $f$ in the coarse moduli spaces, $\pi:F\rightarrow G$, is strongly representable, also $f$ is strongly representable.
\end{lem}

\begin{proof}
We must show that, given a scheme $B$ with a morphism to $\mathcal G$, the fiber product of $f$ with this morphism, $\mathcal F_B$, is a scheme. 
\begin{equation*}
\xymatrix{
{\mathcal F_B} \ar[d] \ar[r] &{\mathcal F} \ar[r] \ar[d]^{f} & {F} \ar[d]^{\pi}\\
{B} \ar[r] & {\mathcal G} \ar[r] & {G}
}
\end{equation*}
Since $f$ is representable, we know that $\mathcal F_B$ is an algebraic space, so to show that it is indeed a scheme it is enough to show that there is a projective morphism from $\mathcal F_B$ to a scheme (see \cite{vie} 9.4).
Consider the fiber product of the induced morphism from $B$ to $G$ with $\pi$, $F_B$. Since, by hypothesis, $\pi$ is representable, $F_B$ is a scheme and is endowed with a natural morphism to $\mathcal F_B$, $\rho:\mathcal F_B\rightarrow  F_B$, the base change over $B$ of the map from $\mathcal F$ to $F$. Since $F$ is the coarse moduli space of $\mathcal F$, this map is proper (see \cite{vistoli} 2.1), so also $\rho$ is proper. Now, to show that $\rho$ is projective it is enough to see that it has finite fibers, which follows from the fact that the stacks are Deligne-Mumford.
\end{proof}

\begin{prop} \label{propdm}
The quotient stack $\overline{\mathcal P}_{d,g}^{Ner}$ is Deligne-Mumford for every $d\in \Z$ and for every $g\ge 2$ and is strongly representable over $\overline{\mathcal M}_g^ d$.
\end{prop}

\begin{proof}
The fact that $\overline{\mathcal P}_{d,g}^{Ner}$ is Deligne-Mumford comes from the well known fact that
a quotient stack is Deligne-Mumford if and only if the action of the group on the scheme is GIT-geometric, that is, if all stabilizers are finite and reduced. Since $U_d$ is the locus of curves where balanced line bundles are necessarily stably balanced, the Hilbert point of a $d$-general curve is GIT-semistable if and only if it is GIT-stable, so the GIT-quotient of $U_d$ by $G$ is geometric.

The proof of the strong representability of the natural map from $\overline{\mathcal P}_{d,g}^{Ner}$ to $\overline{\mathcal M}_g^ d$ consists on two steps: first we prove that it is representable and then we use it to prove strong representability.

To prove representability it is sufficient to see that given any section of our quotient stacks over the spectrum of an algebraically closed field $k^\prime$, the automorphism group of it injects into the automorphism group of its image in $\overline{\mathcal M}_g^d$ (see for example \cite{av} 4.4.3). But a section of our quotient stack over an algebraically closed field consists of a map onto a orbit of the action of $G$ in $U_d$. So, the automorphism group of that section is isomorphic to the stabilizer of the orbit. The image of our section consists of a stable curve $X$: the stable model of the projective curve associated to that orbit. As this must be $d$-general, it is GIT-stable and we can use \cite{cap} section 8.2 to conclude that the stabilizer of the orbit injects into the automorphism group of $X$.

So, the map from $\overline{\mathcal P}_{d,g}^{Ner}$ to $\overline{\mathcal M}_g^ d$ is representable. It follows now immediately that it is also strongly representable from Lemma \ref{rep} and the fact that the GIT-quotients yielding $\overline P_{d,g}^{Ner}$ and $\overline{M}_g$ are geometric (see Example \ref{coarse}).
\end{proof}

\begin{defi}\label{functorial}
Let $f:\mathcal X\rightarrow S$ be a family of stable curves. A compactification of the relative Picard variety of degree $d$ associated to $f$ is a projective $S$-scheme $P$ whose fiber over closed points $\xi$ of $S$ is isomorphic to $\phi_d^{-1}(X_\xi)$, where by $X_{\xi}$ we mean the fiber of $f$ over $\xi$.
\end{defi}

The following is an immediate consequence of the previous Proposition.

\begin{cor}\label{functorialneron}
The Deligne-Mumford stack $\overline{\mathcal P}_{d,g}^{Ner}$ gives a functorial way of getting compactifications of the relative Picard variety of degree $d$ for families of $d$-general curves in the sense of Definition \ref{functorial}.
\end{cor}

\begin{rem}\label{functorialcap}\upshape{
Denote by
\begin{equation}\label{capstack}
 \overline{\mathcal P}_{d,g}:=[H_d/G]
\end{equation}
the quotient stack of the action of $G=PGL(r+1)$ in $H_d$.
Then, if
$(d-g+1,2g-2)=1$, $\overline{\mathcal P}_{d,g}=\overline{\mathcal P}_{d,g}^{Ner}$ and all we said in this section was already proved in \cite{capneron} section 5 for $\overline{\mathcal P}_{d,g}$.}
\end{rem}

\subsection{N\'eron models of families of $d$-general curves}\label{neron1}

Recall that, given DVR (discrete valuation ring) $R$ with function field $K$ and an abelian variety $A_K$ over $K$, the N\'eron model of $A_K$, $N(A_K)$, is a smooth model of $A_K$ over $B=$Spec$R$ defined by the following universal property (cf. \cite{BLR} Definition 1): for every smooth scheme $Z$ over $B$ with a map $u_K:Z_K\rightarrow A_Z$ of its generic fiber, there exists an unique extension of $u_K$ to a $B$-morphism $u:Z\rightarrow N(A_K)$. Note that $N(A_K)$ may fail to be proper over $B$ but it is always separated.

Let $f:\mathcal X\rightarrow B$ be a family of stable curves with $\mathcal X$ nonsingular. Denote by $X_k$ the closed fiber of the family and by $\mathcal X_K$ its generic fiber. The question is how to construct the N\'eron model of the Picard variety Pic$^d\mathcal X_K$ in a functorial way over $\overline{\mathcal M}_g$. Even if it is natural to look at the Picard scheme (of degree $d$) of the family, Pic$^d_f\rightarrow B$, which is smooth and has generic fiber equal to Pic$^d\mathcal X_K$, it turns out to be non satisfactory since it fails to be separated over $B$ if the closed fiber $X_k$ of $f$ is reducible.

Denote by $\mathcal P_{d,g}^{Ner}$ the quotient stack $[U_d^{st}/G]$, where $U_d^{st}$ is the locus of points in $U_d$ parametrize $d$-general stable curves.

It is clear that the statement of Proposition \ref{propdm} holds for $\mathcal P_{d,g}^{Ner}$ since $U_d^{st}$ is a $G$-invariant subscheme of $U_d$.
So, given a family of $d$-general stable curves $f:\mathcal X\rightarrow B$, the fiber product $\mathcal P_{d,g}\times_{\overline{\mathcal M}_g^d}B$, where $B\rightarrow \overline{\mathcal M}_g$ is the moduli map associated to the family $f$, is a scheme over $B$, denoted by $P_f^d$.

\begin{equation*}
\xymatrix{
{ P_f^d} \ar[d] \ar[r] & {{\mathcal P}_{d,g}^{Ner}} \ar[d] \\
{B} \ar[r] & {\overline{\mathcal M}_g^d}
}
\end{equation*}
Suppose $\mathcal X$ is regular. Then, from \cite{capneron}, Theorem 6.1, we get that $P_f^d\cong N($Pic$^d\mathcal X_K)$.

\subsection{Combinatorial description of $d$-general curves in $\overline M_g$}\label{dgeneraldesc}

Recall the notions of $d$-general and $d$-special curve from
Definition \ref{dgeneral}.

Following the notation of \cite{capest}, we will denote by $\Sigma_g^d$ the locus in $\overline M_g$ of $d$-special curves. So, $\Sigma_g^d$ consists of stable curves $X$ of genus $g$ such that $\tilde B_X^d\setminus B_X^d\ne \emptyset$ (see Definition \ref{balanced}). In particular, $\Sigma_g^d$ is contained in the closed subset of $\overline M_g$ consisting of reducible curves. Let us also denote by $\overline{M}_g^d$ the locus of $d$-general genus $g$ stable curves (so $\Sigma_g^d\cup \overline{M}_g^d=\overline M_g$, for all $d\in \Z$).

 From \cite{cap}, Lemma 6.1, we know that $\overline{M}_g^d$ is the image under $\phi_d$ of $U_d$, so it is an open subset of $\overline M_g$.

Recall that a \textit{vine curve} is a curve with two smooth irreducible components meeting in an arbitrary number of nodes. The closure in $\overline M_g$ of the vine curves of genus $g$ is precisely the locus of reducible curves.

In Proposition \ref{dgeneralprop}, we give a geometric description of $\Sigma_g^d$.

\begin{ex}\upshape{
Let $d=1$. From \cite{capest}, Prop. 3.15 we know that, if $g$ is odd, $\Sigma_g^1$ is empty and that if $g$ is even, $\Sigma_g^1$ is the closure in $\overline M_g$ of the locus of curves $X=C_1\cup C_2$, with $C_1$ and $C_2$ smooth of the same genus and $\sharp (C_1\cap C_2)=k$ odd.
}
\end{ex}

Observe that, from the above example, we get that $\Sigma_g^1$ is the closure in $\overline M_g$ of the $1$-special vine curves of genus $g$. In what follows we will see that this is always the case for any degree.

\begin{lem}\label{vine}
Let $d$ be an integer greater or equal to $1$. Then $\Sigma_g^d$ is the closure in $\overline M_g$ of the locus of $d$-special vine curves.
\end{lem}

\begin{proof}
Let $X$ be a genus $g$ $d$-special curve. As $X$ is stable, using Remark \ref{rembalanced} (C),
this means that there is a connected proper subcurve $Z$ of $X$ and a balanced line bundle $L$ on $X$ such that deg$_ZL=m_Z(d)$.

So, let $Z$ be a connected proper subcurve of $X$ such that deg$_ZL=m_Z(d)$ and such that $w_Z$ is maximal among the subcurves satisfying this relation.
The complementary curve of $Z$ in $X$, $Z^\prime$ must be such that
$$\mbox{deg}_{Z^\prime}L=d-\mbox{deg}_ZL=d-m_Z(d)=M_{Z^\prime}(d).$$
Let us see that $Z^\prime$ is connected as well.

By contradiction, suppose $Z^\prime=Z_1^\prime\cup\dots\cup Z_s^\prime$ is a union of connected components with $s>1$. As deg$_{Z^\prime}L=M_{Z^\prime}(d)$, also each one of its connected components $Z_i^\prime$, $i=1,\dots ,s$, must be such that deg$_{Z_i^\prime}L= M_{Z_i^\prime}(d)$.
In fact, suppose one of them, say $Z_j^\prime$, is such that deg$_{Z_j^\prime}L<M_{Z_j^\prime}(d)$. Then,
$$\mbox{deg}_{Z^\prime}L=\sum_{i=1}^s\mbox{deg}_{Z_i^\prime}L<\sum_{i=1}^s\frac{dw_{Z_i^\prime}}{2g-2}+\frac{k_{Z_i^\prime}}2=\frac{dw_{Z^\prime}}{2g-2}+\frac{k_{Z^\prime}}2=M_{Z^\prime}(d)$$
leading us to a contradiction. Note that the sum of the $k_{Z_i^\prime}$'s is $k_{Z^\prime}$ because, being the $Z_i^\prime$'s the connected components of $Z^\prime$, they do not meet each other.

Now, let us consider $W:=Z\cup Z_1^\prime$. As $s>1$, $W$ is a connected proper subcurve of $X$ with $w_W=w_Z+w_{Z_1^\prime}>w_Z$. Indeed, as $X$ is stable,
\begin{equation}\label{exc}
0<w_Y<2g-2
\end{equation}
for every proper subcurve $Y$ of $X$, since there are no exceptional components. Moreover,
$$\mbox{deg}_WL=\mbox{deg}_ZL+\mbox{deg}_{Z_1^\prime}L=\frac{dw_Z}{2g-2}-\frac{k_Z}2+\frac{dw_{Z_1^\prime}}{2g-2}+\frac{k_{Z_1^\prime}}2=\frac{dw_W}{2g-2}-\frac{k_W}{2}$$
because, being $Z_1^\prime$ a connected component of $Z^\prime$, we have that $$k_Z-k_{Z_1^\prime} = k_Z - \sharp(Z\cap Z_1^\prime)=k_W.$$ So, $W$ is a connected proper subcurve of $X$ with deg$_WL=m_W(d)$ and with $w_W>w_Z$.
This way, we achieved a contradiction by supposing that $Z^\prime$ is not connected.

As both $Z$ and $Z^\prime$ are limits of smooth curves and $\Sigma_g^d$ is closed in $\overline M_g$, then $X$ lies in the closure in $\overline M_g$ of the locus of genus $g$ $d$-special vine curves.
\end{proof}

Given integers $d$ and $g$, we will use the following notation to indicate greatest common divisor
$$G_d:=(d-g+1,2g-2).$$
From \cite{cap}, we know that $\Sigma_g^d$ is a proper closed subset of $\overline M_g$ and that $\Sigma_g^d=\emptyset$ if and only if $G_d=1$ (see Prop. 6.2 of loc. cit.).

\begin{rem}\label{special}\upshape{
From Lemma \ref{vine} and Remark \ref{rembalanced}(D) we conclude that a stable curve $X$ is $d$-special if and only if there is a connected proper subcurve $Z$ of $X$ such that $\overline{X\setminus Z}$ is connected and $\frac{2g-2}{G_d}$ divides $w_Z$.}
\end{rem}

\begin{rem}\upshape{
If $G_d=2g-2$, which means that $d\equiv (g-1)($mod $2g-2)$, an immediate consequence of the previous Remark is that all reducible curves are $(g-1)$-special. This is the opposite situation to the case $G_d=1$.}
\end{rem}

From Remark \ref{special} we see that $\Sigma_g^d$ depends only on $G_d$. This is evident in the following proposition, where we give a geometric description of $\Sigma_g^d$.

\begin{prop}\label{dgeneralprop}
Let $d$ be an integer greater or equal to $1$. Then $\Sigma_g^d$ is the closure in $\overline M_g$ of vine curves $X=C_1\cup C_2$ such that
$$\frac{2g-2}{G_d}\mid w_{C_1}.$$
More precisely, $\Sigma_g^d$ is the closure in $\overline M_g$ of the following vine curves:\\
given integers $m$ and $k$ with
$$1\le m< G_d$$
and
$$1\le k\le \mbox{min}\{\frac{2g-2}{G_d}m+2,2g-\frac{2g-2}{G_d}m\}\mbox{, }k\equiv \frac{2g-2}{G_d}m(\mbox{mod }2),$$
then
$X=C_1\cup C_2$, with $\sharp(C_1\cap C_2)=k$,
and
\begin{itemize}
\item $g(C_1)=\frac{g-1}{G_d}m-\frac{k}{2}+1$;
\item $g(C_2)=g-\frac{g-1}{G_d}-\frac{k}{2}$.
\end{itemize}
\end{prop}

\begin{proof}
The first part of the proposition is an immediate consequence of Lemma \ref{vine} and Remark \ref{special}.

Now, let $X$ be a $d$-special genus $g$ vine curve $X=C_1\cup C_2$ with $\sharp(C_1\cap C_2)=k$.
From Remark \ref{special} we know that there exists an integer $m$ such that
$$m\frac{2g-2}{G_d}=w_{C_1}$$
with $1\le m<G_d$ because, as $X$ is a stable curve, $\frac{w_{C_1}}{2g-2}$ must be smaller than $1$.

As $w_{C_1}=2g(C_1)-2+k$, we get that $k\equiv \frac{2g-2}{G_d}m($mod $2)$ and that
$$g(C_1)=\frac{g-1}{G_d}m-\frac{k}{2}+1.$$
Now, as $g=g(C_1)-g(C_2)+k-1$, we get that
$$g(C_2)=g-\frac{g-1}{G_d}-\frac{k}{2}.$$
As $g(C_1)$ and $g_(C_2)$ must be greater or equal than $0$, we get,
respectively, that
$$k\le \frac{2g-2}{G_d}m+2\mbox{ and }k\le 2g-\frac{2g-2}{G_d}m.$$
It is easy to see that if $g(C_1)$ or $g(C_2)$ are equal to $0$ then $k\ge 3$. So, the vine curves we constructed are all stable.
\end{proof}

\begin{rem}\upshape{
Since by smoothing a vine curve in any of its nodes we get an irreducible curve, we see that the above set of \av generators\av of $\Sigma_g^d$ is minimal in the sense that none of them lies in the closure of the others.}
\end{rem}

The dependence of $\Sigma_g^d$ on $G_d$ gets even more evident in the following proposition.

\begin{prop} \label{strat}
For every $d, d^\prime \in \Z$, $G_d|G_{d^\prime}$ if and only if $\Sigma_g^d\subset\Sigma_g^{d^\prime}$.
\end{prop}

\begin{proof}
That $G_d|G_{d^\prime}$ implies that $\Sigma_g^d\subset\Sigma_g^{d^\prime}$ is immediate from Remark \ref{special}.

Now, suppose $\Sigma_g^d\subset\Sigma_g^{d^\prime}$. If $G_d=1$ then obviously $G_d|G_{d^\prime}$. For $G_d\ne 1$ we will conclude by contradiction that $G_d|G_{d^\prime}$. So, suppose $G_{d^\prime}\nmid G_d$. Then, also $\frac{2g-2}{G_d^\prime}\nmid \frac{2g-2}{G_d}$. We will show that there exists a stable curve $X$ consisting of two smooth irreducible components $C_1$ and $C_2$ meeting in $\delta$ nodes ($\delta\ge 1$) which is $d$-special but not $d^\prime$-special.

Take $X$ such that $w_{C_1}=\frac{2g-2}{G_d}$. If such a curve exists and is stable then we are done because $X$ will clearly be $d$-special and not $d^\prime$-special. In fact, by construction, $\frac{2g-2}{G_{d^\prime}}$ does not divide $w_{C_1}$ and $\frac{2g-2}{G_{d^\prime}}$ will not divide $w_{C_2}$ too because $w_{C_2}=(2g-2)-\frac{2g-2}{G_d}$.

So, $X$ must be such that
\begin{itemize}
\item $g(C_1)=\frac{g-1}{G_d}+1-\frac{\delta}{2}$
\item $g(C_2)=g-\frac{g-1}{G_d}-\frac{\delta}{2}$
\item $\delta\ge 1$ and $\delta\equiv\frac{2g-2}{G_d}($mod $2)$.
\end{itemize}
As $g(C_i)$ must be greater or equal than $0$ and the curve $X$ must be stable, we must check if such a construction is possible.

So, if $\frac{2g-2}{G_d}\equiv 1($mod $2)$, take $\delta=1$. Then we will have that $g(C_1)=\frac{g-1}{G_d}+\frac{1}{2}$ and $g(C_2)=g-\frac{g-1}{G_d}-\frac{1}{2}$, which are both greater than $1$ because we are considering $G_d> 1$.

If $\frac{2g-2}{G_d}\equiv 0($mod $2)$, take $\delta=2$. Then we will have that $g(C_1)=\frac{g-1}{G_d}$ and $g(C_2)=g-\frac{g-1}{G_d}-1$, again both greater than $1$. We conclude that $X$ is a stable curve.
\end{proof}

The following is immediate.

\begin{cor}
For all $d$ and $d^\prime$, $\Sigma_g^d=\Sigma_g^{d^\prime}$ if and only if $G_d=G_{d^\prime}$.
\end{cor}

For each positive divisor $M$ of $2g-2$ there is an integer $d=M+g-1$ such that $G_d=M$. So, for each such $M$, we can define
$$\Sigma_{g,M}:=\Sigma_g^d \mbox{  and  } \overline{M}_g^M=\overline M_g\setminus \Sigma_{g,M}.$$

For example, $\overline{M}_g^{2g-2}$ consists of irreducible curves and $\overline{M}_{g}^1=\overline{M}_g$.

The following is now immediate.

\begin{prop}\label{lattice}
The open subsets $\overline{M}_g^M$ associated to the positive
divisors $M$ of $2g-2$, form a lattice of open subschemes of
$\overline M_g$ such that $\overline{M}_g^M\subset
\overline{M}_g^{M^\prime}$ if and only if $M^\prime|M$.
\end{prop}

\section{Modular description of Balanced Picard stacks over $\overline{\mathcal M}_g$} \label{modular}

Suppose $(d-g+1,2g-2)=1$. Then $\overline{\mathcal M}_g^d=\overline{\mathcal M}_g$ and $\overline{\mathcal P}_{d,g}^{Ner}=\overline{\mathcal P}_{d,g}$ (see section \ref{dgeneraldesc}). Moreover, from \cite{capneron}, 5.10, we know that $\overline{\mathcal P}_{d,g}$ is the \av rigidification\uv in the sense of \cite{avc} (see section \ref{rigidification} below) of the category whose sections over a scheme $S$ are pairs $(f:\mathcal X\rightarrow S,\mathcal L)$ where $f$ is a family of quasistable curves of genus $g$ and $\mathcal L$ is a balanced line bundle on $\mathcal X$ of relative degree $d$. Arrows between such pairs are given by cartesian diagrams
\begin{equation*}
\xymatrix{
{\mathcal X} \ar[d]_{f} \ar[r]^{h} & {\mathcal X^\prime} \ar[d]^{f^\prime} \\
{S} \ar[r] & {S^\prime}
}
\end{equation*}
and an isomorphism $\mathcal L\cong h^*\mathcal L^\prime\otimes f^*M$, for some $M\in$ Pic $S$.

This description uses heavily the existence of Poincar\'e line bundles for families of quasistable curves, established in loc. cit., Lemma 5.5. However, this works only if $(d-g+1,2g-2)=1$.

In order to overcome this difficulty we will try to define the stack of line bundles of families of stable curves.

We will start by recalling the definition of \av Picard stack
associated to a morphism of schemes\uv. Roughly speaking, a
\textit{Picard stack} is a stack together with an \av addition\uv
operation which is both associative and commutative. The theory of
Picard stacks is developed by Deligne and Grothendieck on Section 1.4 of Expos\'e
XVIII in \cite{SGA4}. We will not include here the precise
definition but we address the reader to [ibid.], \cite{Laumon} 14.4
and \cite{behfan}, section 2.

Given a scheme $X$ over $S$ with structural morphism $f:X\rightarrow S$, the $S$-stack of (quasi-coherent) invertible $\mathcal O_X$-modules, $\mathcal Pic_{X/S}$, is a Picard stack: the one associated to the complex of length one
$$\tau_{\leq 0}(Rf_*\mathbb G_m[1]).$$
So, given an $S$-scheme $T$, $\mathcal Pic_{X/S}(T)$ is the groupoid whose objects are invertible $\mathcal O_{X_T}$-modules and whose morphisms are the isomorphisms between them (notation as in \ref{picard}).

$\mathcal Pic_{X/S}$ fits in the exact sequence below, where, given an $S$-scheme $T$, $Pic_{X/S}(T)$ is defined as Pic $X_T/f_T^*($Pic $ T)$ and $B\mathbb G_m(T)$ is the group of line bundles over $T$.
$$0\rightarrow B\mathbb G_m\rightarrow \mathcal Pic_{X/S}\rightarrow Pic_{X/S}\rightarrow 0$$

Now, let us consider the forgetful morphism of stacks $\pi:\overline{\mathcal M}_{g,1}\rightarrow \overline{\mathcal M}_g$. The morphism $\pi$ is strongly representable since, given a morphism $Y$ with a map $h:Y\rightarrow \overline{\mathcal  M}_g$,
the fiber product $Y\times_{\overline{\mathcal M}_g}\overline{\mathcal M}_{g,1}$ is isomorphic to the image of $Id_Y$ under $h$, which is a family of stable curves of genus $g$, say $\mathcal C\rightarrow Y$.

So, we define the category $\mathcal Pic_{\overline{\mathcal M}_{g,1}/\overline{\mathcal M}_g}$ associated to $\pi$ as follows.
Given a scheme $Y$, morphisms from $Y$ to $\overline{\mathcal M}_g$ correspond to families of stable curves over $Y$. So, the objects of $\mathcal Pic_{\overline{\mathcal M_{g,1}}/\overline{\mathcal M}_g}(Y)$ are given by pairs $(\mathcal C\rightarrow Y,\mathcal L)$ where $\mathcal C\rightarrow Y$ is the family of stable curves of genus $g$ associated to a map $Y\rightarrow \overline{\mathcal M}_g$ and $\mathcal L$ is a line bundle on $\mathcal C \cong Y\times_{\overline{\mathcal M}_g}\overline{\mathcal M}_{g,1}$. Morphisms between two such pairs are given by cartesian diagrams
\begin{equation}\label{morf}
\xymatrix{
{\mathcal C} \ar[d] \ar[r]^{h} & {\mathcal C^\prime} \ar[d] \\
{Y} \ar[r] & {Y^\prime}
}
\end{equation}
together with an isomorphism $\mathcal L\cong h^*\mathcal L^\prime$.

We will now concentrate on the following full subcategory of $\mathcal Pic_{\overline{\mathcal M}_{g,1}/\overline{\mathcal M}_g}$ (and on a compactification of it).

\begin{defi}\label{modulardef}
Let ${\mathcal G}_{d,g}$ (respectively $\overline{\mathcal G}_{d,g}$) be the category whose objects are pairs $(f:\mathcal C\rightarrow Y,\mathcal L)$ where $f$ is a family of stable (respectively quasistable) curves of genus $g$ and $\mathcal L$ a balanced line bundle of relative degree $d$ over $Y$. Morphisms between two such pairs are defined as in $\mathcal Pic_{\overline{\mathcal M}_{g,1}/\overline{\mathcal M}_g}$.
\end{defi}

The aim of the present section is to show that both $\mathcal G_{d,g}$ and $\overline{\mathcal G}_{d,g}$ are algebraic (Artin) stacks. We will do it by directly showing that they are isomorphic to the quotient stacks we are about to define.

Let $$G:=GL(r+1).$$
Recall from the section before that $G$ acts on
$H_d$, the locus of GIT-semi\-sta\-ble points in $Hilb_{\mathbb
P^r}^{dt-g+1}$, by projecting onto $PGL(r+1)$.

Consider also the open subset of $H_d$ parametrize points corresponding to stable curves and denote it by $H_d^{st}$. It is easy to see that $H_d^{st}$ is a $G$-equivariant subset of $H_d$. So, we can consider the quotient stacks $[H_d^{st}/G]$ and $[H_d/G]$. Given a scheme $S$, $[H_d^{st}/G](S)$ (respectively $[H_d/G](S)$) consists of $G$-principal bundles $\phi:E\rightarrow S$ with a $G$-equivariant morphism $\psi:E\rightarrow H_d$ (respectively $\psi:E\rightarrow H_d^{st}$). Morphisms are given by pullback diagrams which are compatible with the morphism to $H_d$ (resp. $H_d^{st}$).

\begin{teo}\label{geomdesc}
The quotient stacks $[H_d^{st}/GL(r+1)]$ and $[H_d /GL(r+1)]$ are isomorphic respectively to ${\mathcal G}_{d,g}$ and $\overline{\mathcal G}_{d,g}$.
\end{teo}

\begin{proof}
Since the proof is the same for both cases we will do it only for $[H_d/GL(r+1)]$.

We must show that, for every scheme $S\in SCH_{k}$, the groupoids $\overline{\mathcal G}_{d,g}(S)$ and $[H_d/G](S)$ are equivalent.

Let $(f:\mathcal X\rightarrow S,\mathcal L)$ be a pair consisting of a family $f$ of quasistable curves and a balanced line bundle $\mathcal L$ of relative degree $d$ on $\mathcal X$. We must produce a principal $G$-bundle $E$ on $S$ and a $G$-equivariant morphism $\psi:E\rightarrow H_d$. Since we can take $d$ very large with respect to $g$ (see Remark \ref{rembalanced} (E)), we may assume
that $f_*(\mathcal L)$ is locally free of rank $r+1=d-g+1$. Then, the frame bundle of $f_*(\mathcal L)$ is a principal $GL(r+1)$-bundle: call it $E$.
Now, to find the $G$-equivariant morphism to $H_d$, consider the family $\mathcal X_E:=\mathcal X\times_S E$ polarized by $\mathcal L_E$, the pullback of $\mathcal L$ to $\mathcal X_E$. $\mathcal X_E$ is a family of quasistable curves of genus $g$ and $\mathcal L_E$ is balanced and relatively very ample.
By definition of frame bundle, $f_{E*}(\mathcal L_E)$ is isomorphic to $\mathbb C^{(r+1)}\times E$, so that $\mathcal L_E$ gives an embedding over $E$ of $\mathcal X_E$ in $\P^r\times E$. By the universal property of the Hilbert scheme $H$, this family determines a map $\psi:E\rightarrow H_d$. It follows immediately that $\psi$ is a $G$-equivariant map.

Let us check that isomorphisms in  $\overline{\mathcal G}_{d,g}(S)$ leads canonically to isomorphisms in $[H_d/G](S)$.

An isomorphism between two pairs $(f:\mathcal X\rightarrow S,
\mathcal L)$and $(f^\prime:\mathcal X^\prime\rightarrow S,\mathcal
L^\prime)$ consists of an isomorphism $h:\mathcal X\rightarrow
\mathcal X^\prime$ over $S$ and an isomorphism of line bundles
$\mathcal L\cong h^*\mathcal L^\prime$.
\begin{equation*}
\xymatrix{
{\mathcal X} \ar[rr]^{h} \ar[dr]_{f}& & {\mathcal X^\prime} \ar[dl]^{f^\prime}\\
& {S}
}
\end{equation*}

These determine an unique isomorphism between $f_*(\mathcal L)$ and
$f^\prime_*(\mathcal L^\prime)$ as follows
$$f_*(\mathcal L)\cong f_*(h^*\mathcal L^\prime) \cong f_*^\prime ( h_*(h^*\mathcal L^\prime)))\cong f_*^\prime (\mathcal L^\prime ).$$
As taking the frame bundle gives an equivalence between the category of vector bundles of rank $r+1$ over $S$ and the category of principal $GL(r+1)$-bundles over $S$, the isomorphism $f_*(\mathcal L)\cong f_*^\prime(\mathcal L^\prime)$ leads to an unique isomorphism between their frame bundles, call them $E$ and $E^\prime$ respectively. This isomorphism must be compatible with the $G$-equivariant morphisms $\psi:E\rightarrow H_d$ and $\psi^\prime: E^\prime \rightarrow H_d$ because they are determined by the induced curves $\mathcal X_E$ and $\mathcal X^\prime_{E^\prime}$ embedded in $\P^r$ by $\mathcal L_E$ and $\mathcal L^\prime_{E^\prime}$.

Conversely, given a section $(\phi:E\rightarrow S,\psi:E\rightarrow H_d)$ of $[H_d/G]$ over $S$, let us construct a family of quasistable curves of genus $g$ over $S$ and a balanced line bundle of relative degree $d$ on it.

Let $\mathcal C_d$ be the restriction to $H_d$ of the universal family on $Hilb_{\P^r}^{dt-g+1}$. The pullback of $\mathcal C_d$ by $\psi$ gives a family $\mathcal C_E$ on $E$ of quasistable curves of genus $g$ and a balanced line bundle $\mathcal L_E$ on $\mathcal C_E$ which embeds $\mathcal C_E$ as a family of curves in $\P^r$. As $\psi$ is $G$-invariant and $\phi$ is a $G$-bundle, the family $\mathcal C_E$ descends to a family $\mathcal C_S$ over $S$, where $\mathcal C_S=\mathcal C_E/G$. In fact, since $\mathcal C_E$ is flat over $E$ and $E$ is faithfully flat over $S$, $\mathcal C_S$ is flat over $S$ too (see \cite{ega4}, Prop. 2.5.1). 

Now, since the action of $G$ on $\mathcal C_d$ is naturally linearized (see \cite{cap}, 1.4), also the action of $G$ on $E$ can be linearized to an action on $\mathcal L_E$, yielding descent data for $\mathcal L_E$ (\cite{sga8}, Proposition 7.8). Moreover, $\mathcal L_E$ is relatively (very) ample so, using the fact that $\phi$ is a principal $G$-bundle, we conclude that $\mathcal L_E$ descends to a relatively very ample balanced line bundle on $\mathcal C_S$, $\mathcal L_S$ (see proof of Proposition 7.1 in \cite{git}).

It is straightforward to check that an isomorphism on $[H_d/G](S)$ leads to an unique isomorphism in $\overline{\mathcal G}_{d,g}(S)$.

\end{proof}

We will call $\mathcal G_{d,g}$ and $\overline{\mathcal G}_{d,g}$ respectively \textit{balanced Picard stack} and \textit{compactified balanced Picard stack}. The relation between $\mathcal G_{d,g}$ and $\overline{\mathcal G}_{d,g}$ and the stacks $\mathcal P_{d,g}$ and $\overline{\mathcal P}_{d,g}$ defined in \cite{capneron} will be clear in the following section.

\begin{rem} \upshape{
Since $\mathbb G_m$ is always included in the stabilizers at every point of the action of $G$ both in $H_d$ and in $H_d^{st}$, the quotient stacks above are never Deligne-Mumford. However, they are, of course Artin stacks with a presentation given by the schemes $H_d^{st}$ and $H_d$, respectively.

Notice also that, since the scheme $H_d$ is nonsingular and closed (see Lemma 2.2 in \cite{cap}), the algebraic stack $\overline{\mathcal G}_{d,g}$ is a smooth compactification of $\mathcal G_{d,g}$.}
\end{rem}

Let
$$d\mathcal G_{d,g}$$
be the category over $SCH_ k$ whose sections over a scheme $S$, $d\mathcal G_{d,g}(S)$, consists of pairs $(f:\mathcal X\rightarrow S, \mathcal L)$, where $f$ is a family of $d$-general quasistable curves of genus $g$ and $\mathcal L$ is an $S$-flat balanced line bundle on $\mathcal X$ of relative degree $d$. Arrows between two such pairs are given by cartesian diagrams like in (\ref{morf}).

Using the same proof of Prop. \ref{geomdesc} we conclude that $d\mathcal G_{d,g}$ is isomorphic to the quotient stack $[U_d/G]$.

\section{Rigidified Balanced Picard stacks}\label{rigidification}

Recall that in (\ref{capstack}) we denoted the quotient stack $[H_d/PGL(r+1)]$ by $\overline{\mathcal P}_{d,g}$. Define analogously
$$\mathcal P_{d,g}:=[H_d^{st}/PGL(r+1)]$$
(recall that $H_d^{st}\subset H_d$ parametrizes embedded stable curves).

In what follows we will relate $\mathcal G_{d,g}$ and $\overline{\mathcal G}_{d,g}$, respectively, with $\mathcal P_{d,g}$ and $\overline{\mathcal P}_{d,g}$ using the notion of rigidification of a stack along a group scheme defined by Abramovich, Vistoli and Corti in \cite{avc}, 5.1.

Note that each object $(f:\mathcal X\rightarrow S,\mathcal L)$ in $\overline{\mathcal G}_{d,g}$ have automorphisms given by scalar multiplication by an element of $\Gamma (\mathcal X,\mathbb G_m)$ along the fiber of $\mathcal L$. Since these automorphisms fix $\mathcal X$, there is no hope that our stack $\overline{\mathcal G}_{d,g}$ can be representable over $\overline{\mathcal M}_g$ (see \cite{av}, 4.4.3). The rigidification procedure removes those automorphisms.

More precisely, the set up of rigidification consists of:
\begin{itemize}
\item a stack $\mathcal G$ over a base scheme $\mathbb S$;
\item a finitely presented group scheme $G$ over $\mathbb S$;
\item for any object $\xi$ of $\mathcal G$ over an $\mathbb S$-scheme $S$, an embedding
$$i_\xi:G(S)\rightarrow \mbox{Aut}_S(\xi)$$
compatible with pullbacks.
\end{itemize}
Then the statement (Theorem 5.1.5 in \cite{avc}) is that there exists a stack $\mathcal G^G$ and a morphism of stacks $\mathcal G\rightarrow \mathcal G^G$ over $\mathbb S$ satisfying the following conditions:
\begin{enumerate}
\item For any object $\xi\in \mathcal G(S)$ with image $\eta\in \mathcal G^G(S)$, the set $G(S)$ lies in the kernel of Aut$_S(\xi)\rightarrow$Aut$_S(\eta)$;
\item The morphism $\mathcal G\rightarrow \mathcal G^G$ above is universal for morphisms of stacks $\mathcal G\rightarrow \mathcal F$ satisfying condition $(1)$ above;
\item If $S$ is the spectrum of an algebraically closed field, we have in $(1)$ above that $Aut_S(\eta)=$Aut$_S(\xi)/G(S)$;
\item A moduli space for $\mathcal G$ is also a moduli space for $\mathcal G^G$.
\end{enumerate}
$\mathcal G^G$ is the \textit{rigidification} of $\mathcal G$ along $G$.

By taking $\mathbb S$=Spec $k$, $\mathcal G=\overline{\mathcal G}_{d,g}$ and $G=\mathbb G_m$ we see that our situation fits up in the setting above.
It is easy to see that, since $GL(r+1)$ acts on $H_d$  by projection onto $PGL(r+1)$, $[H_d/PGL(r+1)]$ is the rigidification of $[H_d/GL(r+1)]\cong \overline{\mathcal G}_{d,g}$ along $\mathbb G_m$: denote it by $\overline{\mathcal G}_{d,g}^{\mathbb G_m}$. Naturally, the same holds for $[H_d^{st}/PGL(r+1)]$ and $\mathcal G_{d,g}^{\mathbb G_m}$. The following is now immediate.

\begin{prop} \label{isorig}
The stacks $\mathcal P_{d,g}$ and $\overline{\mathcal P}_{d,g}$ are isomorphic, respectively, to $\mathcal G_{d,g}^{\mathbb G_m}$ and $\overline{\mathcal G}_{d,g}^{\mathbb G_m}$. In particular, $\mathcal G_{d,g}$ and $\overline{\mathcal G}_{d,g}$ are $\mathbb G_m$-gerbes over $\mathcal P_{d,g}$ and $\overline{\mathcal P}_{d,g}$, respectively.
\end{prop}

\begin{rem}\upshape{
The previous proposition holds, of course, also for the rigidification along $\mathbb G_m$ of $d\mathcal G_{d,g}$ , $d\mathcal G_{d,g}^{\mathbb G_m}$, and $\overline{\mathcal P}_{d,g}^{Ner}$, which we have studied in section \ref{dgeneralstack} (see the definition of $d\mathcal G_{d,g}$ in the end of section \ref{modular}).}
\end{rem}

Recall from the beginning of section \ref{modular} that, if $(d-g+1,2g-2)=1$, $\overline{\mathcal P}_{d,g}$ has a modular description as the rigidification for the action of $\mathbb G_m$ in a certain category.

In order to remove from $\overline{\mathcal G}_{d,g}$ the automorphisms given by $\mathbb G_m$, we first consider the auxiliar category $\overline{\mathcal A}_{d,g}$, whose objects are the same of $\overline{\mathcal G}_{d,g}$ but where morphisms between pairs $(\mathcal C\rightarrow Y,\mathcal L)$ and $(\mathcal C^ \prime\rightarrow Y^\prime,\mathcal L^\prime)$ are given by equivalence classes of morphisms in $\overline{\mathcal G}_{d,g}$ by the following relation. Given a cartesian diagram
\begin{equation}
\xymatrix{
{\mathcal C} \ar[d] \ar[r]^{h} & {\mathcal C^\prime} \ar[d] \\
{Y} \ar[r] & {Y^\prime}
}
\end{equation}
and isomorphisms $\phi:\mathcal L\rightarrow h^*\mathcal L^\prime$ and $\psi:\mathcal L\rightarrow h^*\mathcal L^\prime$, we say that $\phi$ is equivalent to $\psi$ if there exists $\alpha\in \mathbb G_m$ such that $\underline{\alpha} \circ \psi=\phi$, where by $\underline \alpha$ we mean the morphism induced by $\alpha$ in $\mathcal L^\prime$ (fiberwise multiplication by $\alpha$).

There is an obvious morphism of $\overline{\mathcal G}_{d,g}\rightarrow\overline{\mathcal A}_{d,g}$ satisfying property $(1)$ above and universal for morphisms of $\overline{\mathcal G}_{d,g}$ in categories satisfying it. However, it turns out that $\overline{\mathcal A}_{d,g}$ is not a stack. In fact, it is not even a prestack since, given an \'etale cover $\left\lbrace \coprod_iY_i\rightarrow Y\right\rbrace $ of $Y$, the natural morphism $\overline{\mathcal A}_{d,g}(Y)$ to $\overline{\mathcal A}_{d,g}(\coprod_iY_i\rightarrow Y)$, the category of effective descent data for this covering, is not fullyfaithful but just faithful.

Let us now consider the category $\overline{\mathcal C}_{d,g}$, with the following modular description. A section of $\overline{\mathcal C}_{d,g}$ over a scheme $S$ is given by a pair $(f:\mathcal X\rightarrow S,\mathcal L)$, where $f$ is a family of quasistable curves of genus $g$ and $\mathcal L$ is a balanced line bundle on $\mathcal X$ of relative degree $d$. Arrows between such pairs are given by cartesian diagrams
\begin{equation*}
\xymatrix{
{\mathcal X} \ar[d]_{f} \ar[r]^{h} & {\mathcal X^\prime} \ar[d]^{f^\prime} \\
{S} \ar[r] & {S^\prime}
}
\end{equation*}
and equivalence classes of isomorphisms $\mathcal L\cong h^*\mathcal L^\prime\otimes f^*M$, for some $M\in$ Pic $S$, for the following relation. Isomorphisms $\phi:\mathcal L\rightarrow h^*\mathcal L^ \prime\otimes f^*M$ and $\psi:\mathcal L\rightarrow h^*\mathcal L^ \prime\otimes f^*N$ are equivalent if there exists an isomorphism $g:N\rightarrow M$ of line bundles on $S$ such that the following diagram commutes.
\begin{equation*}
\xymatrix{
{\mathcal L} \ar[r]^{\phi} \ar[rd]_{\psi} & {h^*\mathcal L^\prime\otimes f^*M} \\
& {h^*\mathcal L^\prime\otimes f^*N} \ar[u]_{id\otimes f^*g}
}
\end{equation*}
Straightforward computations show that $\overline{\mathcal C}_{d,g}$ is a prestack. Moreover, given the \'etale cover $(\coprod_iY_i\rightarrow Y)$ of $Y$, $\overline{\mathcal A}_{d,g}(\coprod_iY_i\rightarrow Y)$ is isomorphic to $\overline{\mathcal C}_{d,g}(\coprod_iY_i\rightarrow Y)$.

So, we conclude that the stackification of $\overline{\mathcal C}_{d,g}$ is the rigidification of $\overline{\mathcal G}_{d,g}$ under the action of $\mathbb G_m$.

\begin{prop}
The stack $\overline{\mathcal P}_{d,g}$ (resp. $\mathcal P_{d,g}$) is the stackification of the prestack whose sections over a scheme $S$ are given by pairs $(f:\mathcal X\rightarrow S,\mathcal L)$, where $f$ is a family of quasistable (resp. stable) curves of genus $g$ and $\mathcal L$ is a balanced line bundle on $\mathcal X$ of relative degree $d$. Arrows between two such pairs are given by cartesian diagrams
\begin{equation*}
\xymatrix{
{\mathcal X} \ar[d]_{f} \ar[r]^{h} & {\mathcal X^\prime} \ar[d]^{f^\prime} \\
{S} \ar[r] & {S^\prime}
}
\end{equation*}
and an isomorphism $\mathcal L\cong h^*\mathcal L^\prime\otimes f^*M$, for some $M\in$ Pic $S$.
\end{prop}

\begin{rem}\upshape{
Let $d>>0$. Then, as in the case $(d-g+1,2g-2)=1$, there is a canonical map from $\overline{\mathcal P}_{d,g}$ and $\mathcal P_{d,g}$ for $\overline{P}_{d,g}$ and $P_{d,g}$, the GIT-quotients of $H_d$ and $H_d^{st}$, respectively, by the action of $PGL(r+1)$. If $(d-g+1,2g-2)\neq 1$, these quotients are not geometric, which implies that these maps are universally closed but not separated.  So, $\overline{P}_{d,g}$ and $P_{d,g}$ are not coarse moduli spaces for those stacks since the associated maps from the stacks onto them are not proper. However, at least if the base field has characteristic $0$, we have that the GIT-quotients are \textit{good moduli spaces} in the sense of Alper (see \cite{alper}).}
\end{rem}

\subsection{Rigidified balanced Picard stacks and N\'eron models.}

The main question now is the following: does our balanced Picard
stack ${\mathcal P}_{d,g}$ para\-me\-tri\-zes N\'eron models of
families of stable curves for every $d$ as $\mathcal P_{d,g}^{Ner}$
does for families of $d$-general curves (see \ref{neron1})?

Given a family of stable curves $f:\mathcal X\rightarrow B=$ Spec
$R$, we will denote by $Q_f^d$ the base-change of the map ${\mathcal
P}_{d,g}\rightarrow \overline{\mathcal M}_g$ by the natural map
$B\rightarrow \overline{\mathcal M}_g$: $B\times_{\overline{\mathcal
M}_g}\mathcal P_{d,g}$. Is $Q_f^d$ isomorphic to $N($Pic$^d\mathcal
X_K$?

The problem here is that the map ${\mathcal P}_{d,g}\rightarrow \overline{\mathcal M}_g$ is not representable in general. In fact, if it were, ${\mathcal P}_{d,g}$ would be a Deligne-Mumford stack. Indeed, it is easy to see that a stack with a representable map to a Deligne-Mumford stack is necessarily Deligne-Mumford.
As we already mentioned, ${\mathcal P}_{d,g}$ is not Deligne-Mumford in general since it is the quotient stack associated to a non-geometric GIT-quotient.

As a consequence of this, if the closed fiber of $f$ is not $d$-general, then $ Q_f^d$ is not even equivalent to an algebraic space. In fact, from a common criterion for representability (see for example \cite{av}, 4.4.3 or the proof of Proposition \ref{propdm}), we know that $ Q_f^d$ would be equivalent to an algebraic space if and only if the automorphism group of every section of ${\mathcal P}_{d,g}$ over $k$ with image in $\overline{\mathcal M}_g$ isomorphic to $ X_k\rightarrow k$ injects into the automorphism group of $ X_k$. Since such a section corresponds to a map onto its orbit in $H_d$ by the action of $PGL(r+1)$, the automorphism group of such a section is isomorphic to the stabilizer of that orbit. So, as $X_k$ is not $d$-general, the stabilizer of its associated orbit in $H_d$ is not finite, which implies that it cannot have an injective morphism to the automorphism group of $ X_k$, which is, of course, finite.

The following example will clarify what we just said.

\begin{ex} \upshape{
Let $d=g-1$ and
$f:\mathcal X\rightarrow B=$ Spec $R$ be a family of stable curves
such that $\mathcal X$ is regular and $ X_k$ is a reducible curve
consisting of two smooth components $C_1$ and $C_2$ of genus $g_1$
and $g_2$ respectively, meeting in one point (of course,
$g=g_1+g_2$). Then, the fiber over $k$ of $ Q_f^d$ is a stack with a
presentation given by a subscheme of the Hilbert scheme $H_d$, consisting of two connected components of dimension $r(r+2)+g$
(see \cite{cap} Example 7.2). These correspond to projective realizations of $ X_k$ on $\mathbb P^r$ given by line bundles with the two possible balanced multidegrees on $ X_k$: $(g_1, g_2-1)$ and $(g_1-1,g_2)$. By Proposition 5.1 of \cite{cap}, given a point $h$ in one of these components, there is a point in $\overline{O_{PGL(r+1)}(h)}$ representing the quasistable curve with stable model $ X_k$ embedded by a line bundle of multidegree $(g_1-1, 1, g_2-1)$.

So, as a stack, the fiber of $ Q_f^d$ over $k$ is reducible but the
GIT quotient of the Hilbert scheme presenting it by the action of
$PGL(r+1)$ is irreducible and isomorphic to the Jacobian of $ X_k$.
As a consequence, $Q_f^d$  can never be isomorphic to the N\'eron
model $N($Pic$^d\mathcal X_K)$.

This is an example of a situation where the GIT-quotient identifies
two components of the Hilbert scheme while in the quotient stack
these two components remain separated.
}\end{ex}

So, we have.

\begin{prop} 
Let $f:\mathcal X\rightarrow B=$ Spec $R$ be a family of stable
curves with $\mathcal X$ regular. Then, using the notation above,
$Q_f^d \cong N$(Pic$^d\mathcal X_K)$ if and only if and only if
$X_k$ is a $d$-general curve.
\end{prop}

\subsubsection{Functoriality for non $d$-general curves}

Let $f:\mathcal X\rightarrow S$ be a family of stable curves. Denote by $\overline P_f^d$ the fiber product of $\overline{\mathcal P}_{d,g}$ by the moduli map of $f$, $\mu_f:S\rightarrow \overline{\mathcal M}_g$.

Recall that, if $(d-g+1,2g-2)=1$, ${\overline P}_f^d$ is a compactification of the relative degree $d$ Picard variety associated to $f$ in the sense of \ref{functorial} (see Remark \ref{functorialcap}).

Let now $(d-g+1,2g-2)\neq 1$. Then, since $\overline{\mathcal P}_{d,g}$ is not representable over $\overline{\mathcal M}_g$, we have just observed that the same cannot be true in general.

However, we have the following result.

\begin{prop}\label{propcan}
Notation as before.
Then $\overline{P}_f^d$ has a canonical proper map onto a compactification $P$ of the relative degree $d$ Picard variety associated to $f$.
\end{prop}

\begin{proof}
If all fibers of $f$ are $d$-general, then from Corollary \ref{functorialneron} it follows that $\overline{P}_f^d$ is a scheme and it gives a compactification of the relative degree $d$ Picard variety associated to $f$.

Suppose not all fibers of $f$ are $d$-general. Then $\overline
P_f^d$ is a stack with a presentation given by the subscheme $H_d^f$
of $H_d$ corresponding to the closure in $H_d$ of the locus
parametrize curves isomorphic to the fibers of $f$. $H_d^f$ is
naturally invariant for the action of $PGL(r+1)$ on it. Let $P$ be
the $GIT$-quotient of $H_d^f$ by $PGL(r+1)$. $P$ gives a
compactification of the relative degree $d$ Picard variety
associated to $f$. The proper map $H_d^f\rightarrow P$ factorizes
through $H_d^f\rightarrow \overline{P}_f^d$, the presentation map.
In fact, even if $P$ is not a coarse moduli space for
$\overline{P}_f^d$, there is a canonical map from $\overline{P}_f^d$
onto $P$ which is universal for morphisms of $\overline{P}_f^d$ into
schemes (see \cite{vistoli} section 2). Now, since the map
$H_d^f\rightarrow P$ is proper, then $\overline{P}_f^d\rightarrow P$
must be proper as well.

\end{proof}

\section*{acknowledgements}
I am very grateful to my Ph.D. advisor Professor Lucia Caporaso for introducing me the problem and for following this work very closely.
I am also grateful to
Simone Busonero and Filippo Viviani for useful conversations.

This work was partially supported by a Funda\c{c}\~ao Calouste Gulbenkian fellowship.

\bigskip

\noindent Margarida Melo\\ Dipartimento di Matematica, Universit\`a di Roma Tre - \\
Departamento de Matem\'atica, Universidade de Coimbra.\\
E-mail: \textsl{melo@mat.uniroma3.it}, \textsl {mmelo@mat.uc.pt}.

\end{document}